\documentclass[a4paper]{amsart}
\usepackage[english]{babel}
\usepackage[utf8x]{inputenc}
\usepackage[T1]{fontenc}
\usepackage{amsthm}
\usepackage{amssymb}
\usepackage{amsmath}
\theoremstyle{plain}
\usepackage{chngcntr}
\usepackage{relsize}
\usepackage{pdfpages}

\newtheorem{Lemma}{Lemma}
\newtheorem{Theorem}[Lemma]{Theorem}
\newtheorem{Proposition}[Lemma]{Proposition}
\newtheorem{Corollary}[Lemma]{Corollary}

\newtheoremstyle{case}{}{}{}{}{}{:}{ }{}
\theoremstyle{case}

\newtheoremstyle{cese}{}{}{}{}{}{:}{ }{}
\theoremstyle{cese}

\usepackage[a4paper,top=3cm,bottom=2cm,left=3cm,right=3cm,marginparwidth=1.75cm]{geometry}

\usepackage{float}
\usepackage{nccmath}
\usepackage{mathtools}
\usepackage{amsfonts}
\usepackage{amsmath}
\usepackage{graphicx}
\usepackage[colorinlistoftodos]{todonotes}
\usepackage[colorlinks=true, allcolors=blue]{hyperref}
\newtheoremstyle{case}{}{}{}{}{}{:}{ }{}
\theoremstyle{case}

\title{Character sums to prime power moduli evaluated at binary quadratic forms}



\author{Stephan Baier}
\address{Stephan Baier,
Ramakrishna Mission Vivekananda Educational and Research Institute, Department of Mathematics, G. T. Road, PO Belur Math, Howrah, West Bengal 711202, India}
\email{stephanbaier2017@gmail.com}
\author{Aishik Chattopadhyay}
\address{Aishik Chattopadhyay,
Ramakrishna Mission Vivekananda Educational and Research Institute, Department of Mathematics, G. T. Road, PO Belur Math, Howrah, West Bengal 711202, India}
\email{aishik.ch@gmail.com}
\begin{document}
\maketitle
\begin{abstract} We establish estimates for short character sums to prime power moduli evaluated at binary quadratic forms. This complements estimates established by Heath-Brown for such character sums to squarefree moduli.  Our approach uses $p$-adic analysis. More precisely, we use tools from the $p$-adic theory of exponential sums, as initiated by Mili\'cevi\'c.  
\end{abstract}

\tableofcontents

\section{Introduction and main results}
\subsection{General notations}
Throughout this article, we fix a prime $p>3$ and an arbitrarily small $\varepsilon>0$. For $x\in \mathbb{R}$, we write $e(x):=e^{2\pi i x}$. If $\Omega:\mathbb{R}\rightarrow \mathbb{C}$ is a Schwartz class function, we denote by $\hat{\Omega}$ its Fourier transform, defined as 
$$
\hat\Omega(y):=\int\limits_{\mathbb{R}} \Omega(x)e(-xy){\rm d}x.
$$
By $\log_p(1+x)$, we denote the $p$-adic logarithm of $1+x$, defined for all $x\in p\mathbb{Z}_p$ by the $p$-adic power series expansion
$$
\log_p(1+x)=x-\frac{x^2}{2}+\frac{x^3}{3}-\cdots=\sum\limits_{n=1}^{\infty} (-1)^{n+1}\cdot \frac{x^n}{n}.
$$
We note that the $p$-adic derivative of $\log_p$ satisfies
$$
\log_p' y =\frac{1}{y} \quad \mbox{for} \quad y\in 1+p\mathbb{Z}_p.
$$
\subsection{Review of previous works}
In 2016, Heath-Brown \cite[Theorem 3]{HB} proved the following bound on short character sums evaluated at binary quadratic forms, elaborating on an earlier result by Chang \cite{Cha}.   
\begin{Theorem} \label{HeathBrown}
Let $\varepsilon > 0$ and an integer $r\ge 3$ be given, and suppose that $C\subset \mathbb{R}^2$ is a convex set contained in a disc $\{{\bf x} \in \mathbb{R}^2:
 ||{\bf x}-{\bf x}_0||_2 \le R\}$, $||.||_2$ denoting the Euclidean norm. Let $q \ge 2$ be odd and squarefree, and let $\chi$ be a primitive character to modulus $q$. Then if $Q(x,y)$ is a binary integral quadratic form with $(\det(Q), q) = 1$, we have
\begin{equation} \label{HBsbound}
\sum\limits_{(x,y)\in C}  \chi(Q(x, y)) \ll_{\varepsilon,r} R^{2-1/r}q^{(r+2)/(4r^2)+\varepsilon}
\end{equation}
for $q^{1/4+1/(2r)} \le R \le q^{5/12+1/(2r)}$.
\end{Theorem}

Heath-Brown applied this bound to prove the following result on small solutions of ternary quadratic congruences \cite[Theorem 1]{HB}. 

\begin{Theorem} \label{HeathBrown2} Let $q$ be an odd squarefree integer, and let $\varepsilon>0$ be given. For a quadratic form $Q\in \mathbb{Z}[X_1,X_2,X_3]$ set 
$$
 m(Q;q) := \min\{||x||_2 : x \in \mathbb{Z}^3\setminus \{0,0,0\},\  Q(x) \equiv 0 \bmod{q}\},
$$  
where $||.||_2$ denotes the Euclidean norm. Then for every form $Q\in \mathbb{Z}[X_1,X_2,X_3]$ with $(\det(Q),q) = 1$, the bound 
$$ 
m(Q,q)\ll q^{5/8+\varepsilon}
$$
holds, where the implied constant depends only on $\varepsilon$ and not on $q$ or $Q$. 
\end{Theorem} 

One may conjecture that the above bound holds with an exponent of $1/2$ in place of $5/8$. 
In \cite{BaCh1} and \cite{BaCh2}, the authors of the present paper employed Theorem \ref{HeathBrown} to improve this bound for almost all ternary quadratic forms in certain classes of diagonal ternary quadratic forms, passing the $q^{1/2}$ barrier. 

Heath-Brown's bound \eqref{HBsbound} applies to squarefree integers $q$. In this paper,  we derive bounds that are stronger than \eqref{HBsbound} for prime power moduli, thereby filling a gap in the literature. However, we have not been able to utilize these new estimates to make progress on small solutions of quadratic congruences to prime power moduli.  

Famously, Postnikov \cite{Pos} observed that Dirichlet characters modulo prime powers can be expressed as additive characters. Our approach starts with this conversion. To treat the resulting sums of additive characters evaluated at quadratic forms, we use $p$-adic analysis. More precisely, we apply tools from the $p$-adic theory of exponential sums, as initiated by Mili\'cevi\'c \cite{Mil}, such as $p$-adic stationary phase and $p$-adic Weyl differencing. However, Mili\'cevi\'c's theory of $p$-adic {\it exponent pairs} is not directly applicable in this setting of quadratic forms, which forces us to carry out the relevant steps explicitly, resulting in rather complicated calculations (in particular, with regard to multiplicities of zeros of rational functions). To facilitate applications of the Poisson summation formula, we consider the problem in a smoothed setup. Our method, which is restricted to prime power moduli, is fundamentally different from Heath-Brown's, which combines a variety of tools from the geometry of numbers and algebraic geometry, such as the Riemann hypothesis for curves. Parts of Heath-Brown's work \cite{HB} are inspired by Chang's treatment in \cite{Cha}, other parts are reminiscent of Burgess's works \cite{Bur1} and \cite{Bur2} on short character sums. \\ \\
{\bf Acknowledgements.} The authors would like to thank the anonymous referee for valuable comments. They  would also like to thank the Ramakrishna Mission Vivekananda Educational and Research Institute for an excellent work environment. The research of the second-named author was supported by a CSIR Ph.D fellowship under file number 09/0934(13170)/2022-EMR-I.\\ \\
{\bf Data availability statement.} No data were generated or analyzed in this study.

\subsection{Our main results}
The following results depend on a prime $p$ and Schwartz class functions $\Psi$ and $\Phi$.
Throughout the sequel, we follow the convention that all implied $O$-constants are bounded by $p^KL_{\Psi,\Phi}$, where $K$ is some absolute positive constant and $L_{\Psi,\Phi}$ is a positive constant depending only $\Psi$ and $\Phi$.   
We do not calculate these constants in the present paper. Our first result provides the following conditional bounds.
\begin{Theorem}  \label{mainresult}
Let $p>3$ be a prime and $Q\in \mathbb{Z}_p[X,Y]$ be a quadratic form given as 
$$
Q(x,y)=ax^2+2bxy+cy^2.
$$
Denote by $\det(Q)$ the determinant of $Q$, defined as $\det(Q):=ac-b^2$. Assume that $(c\det(Q),p)=1$. Let $\Psi,\Phi:\mathbb{R}\rightarrow \mathbb{R}_{\ge 0}$ be Schwartz class functions with compact support in $[-1,1]$. Let $q:=p^n$, where $n$ is a positive integer satisfying $n\ge n_0$ for a sufficiently large absolute constant $n_0>0$. Assume that $1\le M,N<p^n$. Let $\chi$ be a primitive character modulo $q$.  Suppose the bounds \eqref{TT1esti} and \eqref{TT2esti} hold for $j_1,j_2\in \mathbb{N}$. Then we have the estimates
\begin{equation} \label{mainestimates} 
\begin{split}
& \sum_{x\in \mathbb{Z}} \sum_{y\in\mathbb{Z}} \Psi\left(\frac{x-A}{M}\right) \Phi\left(\frac{y-B}{N}\right) \chi(Q(x,y))\\ \ll &
\begin{cases} \left(MN^{1/2}+NM^{1/2}\right)q^{(j_1-1)/(2(2j_1-1))} \quad \mbox{if } q^{(j_1-1)/(2j_1-1)}\le M,N\le q^{(3j_1-2)/(2(2j_1-1))}\\ \\ \left(MN^{(7j_2-4)/(2(5j_2-3))}+NM^{(7j_2-4)/(2(5j_2-3))}\right)q^{(j_2-1)/(2(5j_2-3))} \quad \mbox{if } q^{(j_2-1)/(3j_2-2)}\le M,N\le q^{2/3}.
\end{cases}
\end{split} 
\end{equation}
Consequently, if $M=N$, we have 
\begin{equation} \label{maincor} 
\begin{split}
& \sum_{x\in \mathbb{Z}} \sum_{y\in\mathbb{Z}} \Psi\left(\frac{x-A}{N}\right) \Phi\left(\frac{y-B}{N}\right) \chi(Q(x,y))\\ \ll &
\begin{cases} N^{3/2}q^{(j_1-1)/(2(2j_1-1))} \quad \mbox{if } q^{(j_1-1)/(2j_1-1)}\le N\le q^{(3j_1-2)/(2(2j_1-1))}\\ \\ N^{(17j_2-10)/(2(5j_2-3))}q^{(j_2-1)/(2(5j_2-3))} \quad \mbox{if } q^{(j_2-1)/(3j_2-2)}\le N\le q^{2/3}.
\end{cases} 
\end{split}
\end{equation}
\end{Theorem}
Let us compare the two estimates in \eqref{maincor}. We calculate that 
$$
N^{(17j_2-10)/(2(5j_2-3))}q^{(j_2-1)/(2(5j_2-3))} \le N^{3/2}q^{(j_1-1)/(2(2j_1-1))} 
$$
if and only if
$$
N\le q^{(3j_1j_2-j_1-4j_2+2)/((2j_1-1)(2j_2-1))}.
$$
As proved in section \ref{j1j2}, the estimates \eqref{TT1esti} and \eqref{TT2esti} hold for $j_1=3$ and $j_2=4$. In this case, we have 
$$
\frac{j_1-1}{2(2j_1-1)}= \frac{1}{5}, \quad \frac{7j_2-4}{2(5j_2-3)}=\frac{12}{17}, \quad \frac{j_2-1}{2(5j_2-3)}=\frac{3}{34},
$$
$$
\frac{j_1-1}{2j_1-1}=\frac{2}{5}, \quad \frac{3j_1-2}{2(2j_1-1)}=\frac{7}{10}, \quad \frac{j_2-1}{3j_2-2}=\frac{3}{10}, \quad \frac{3j_1j_2-j_1-4j_2+2}{(2j_1-1)(2j_2-1)}=\frac{19}{35}.
$$  
Hence, \eqref{maincor} implies the following unconditional result.

\begin{Corollary}  \label{maincorollary}
Under the assumptions of Theorem \ref{mainresult}, we have the estimates
\begin{equation} \label{maincor1} 
\sum_{x\in \mathbb{Z}} \sum_{y\in\mathbb{Z}} \Psi\left(\frac{x-A}{N}\right) \Phi\left(\frac{y-B}{N}\right) \chi(Q(x,y))\\ \ll \begin{cases}  N^{3/2}q^{1/5} \quad \mbox{if } q^{19/35}\le N\le q^{7/10}\\ \\N^{29/17}q^{3/34} \quad \mbox{if } q^{3/10}\le N< q^{19/35}.
\end{cases} 
\end{equation}
\end{Corollary}

As established in section \ref{j1j2} as well, we may take $j_1=j_2=3$ if $p\equiv \pm 5\bmod{12}$. In this case, we have 
$$
\frac{j_1-1}{2(2j_1-1)}= \frac{1}{5}, \quad \frac{7j_2-4}{2(5j_2-3)}=\frac{17}{24}, \quad \frac{j_2-1}{2(5j_2-3)}=\frac{1}{12},
$$
$$
\frac{j_1-1}{2j_1-1}=\frac{2}{5}, \quad \frac{3j_1-2}{2(2j_1-1)}=\frac{7}{10}, \quad \frac{j_2-1}{3j_2-2}=\frac{2}{7}, \quad \frac{3j_1j_2-j_1-4j_2+2}{(2j_1-1)(2j_2-1)}=\frac{14}{25}.
$$  
Hence, \eqref{maincor} also implies the following unconditional result.

\begin{Corollary}  
If $p\equiv \pm 5\bmod{12}$, then under the assumptions of Theorem \ref{mainresult}, we have the estimates
\begin{equation} \label{maincor2} 
\sum_{x\in \mathbb{Z}} \sum_{y\in\mathbb{Z}} \Psi\left(\frac{x-A}{N}\right) \Phi\left(\frac{y-B}{N}\right) \chi(Q(x,y)) \ll
\begin{cases} N^{3/2}q^{1/5} \quad \mbox{if } q^{14/25}\le N\le q^{7/10}\\ \\ N^{41/24}q^{1/12} \quad \mbox{if } q^{2/7}\le N< q^{14/25}.
\end{cases}
\end{equation}
\end{Corollary}
$ $\\
{\bf Remarks:}\medskip\\ (i) In addition to the assumption $(\det(Q),p)=1$, we also assume that $(c,p)=1$ in our main results above. In section \ref{transf}, this restriction will allow us to diagonalize our quadratic form in such a way that the new variables run over independent intervals. \medskip\\
(ii) In this article, we use a one-dimensional method, i.e., we estimate the sum over one variable non-trivially and then sum up the result over the other variable trivially. It seems conceivable that a two-dimensional method would be capable of improving our result, at the cost of more complicated calculations. This may be considered in future research.   \medskip\\
(iii) We will apply one or two $p$-adic Weyl shifts, followed by Poisson summation, obtaining the two different bounds stated in our results above. The first of them results from one Weyl shift, the second from two Weyl shifts, respectively. Two Weyl shifts are favorable for small values of $N$, as compared to $q$. It is likely that an application of more than two Weyl shifts gives new non-trivial bounds for even smaller values of $N$, but the details of the calculations become more and more complicated as the number of Weyl shifts increases. In particular, it gets harder and harder to control the multiplicities of critical points of certain rational functions arising from the method. The precise effect of a larger number of Weyl shifts may be studied in future research as well.  \medskip\\
(iv) Very short summation intervals of size $\ll q^{\varepsilon}$ are accessible by Korobov's method \cite{Kor}, which is a $p$-adic adaption of Vinogradov's method \cite{Vin}. However, the savings over the trivial bound which one achieves along these lines are very small, whereas the goal of this paper is to obtain maximal savings.   

\subsection{Comparison with Heath-Brown's result}
In the following, we compare our bounds in \eqref{maincor}, \eqref{maincor1} and \eqref{maincor2} for prime power moduli with Heath-Brown's bound in Theorem \ref{HeathBrown} for squarefree moduli. 
Removing the smooth weights using partial summation and applying Theorem \ref{HeathBrown} for $r\ge 3$ and $C$ being a box gives a bound of
\begin{equation} \label{HBbound}
\ll N^{2-1/r}q^{(r+2)/(4r^2)+\varepsilon} \quad \mbox{if } q^{1/4+1/(2r)} \le N \le q^{5/12+1/(2r)}
\end{equation}
for the left-hand side of \eqref{maincor} if $q$ is squarefree. The above $N$-ranges overlap. We calculate that \eqref{HBbound} is optimal in the range 
$$
q^{(r^2+5r+2)/(4(r^2+r))}\le N\le q^{(r^2+3r-2)/(4(r^2-r))}.
$$
We find that 
$$
N^{3/2}q^{(j_1-1)/(2(2j_1-1))} \leq N^{2-1/r}q^{(r+2)/(4r^2)}\quad  \mbox{iff } N\geq q^{r(j_1-1)/((r-2)(2j_1-1))-(r+2)/(2r(r-2))},
$$ 
and therefore our first bound in \eqref{maincor} supersedes \eqref{HBbound} if 
$$
\frac{r(j_1-1)}{(r-2)(2j_1-1)}-\frac{r+2}{2r(r-2)}\le \frac{r^2+5r+2}{4(r^2+r)}.
$$
For our particular case when $j_1=3$, this means that 
$$
N^{3/2}q^{1/5} \leq N^{2-1/r}q^{(r+2)/(4r^2)}\quad  \mbox{iff } N\geq q^{2r/(5(r-2))-(r+2)/(2r(r-2))},
$$ 
and our first bound in \eqref{maincor1} supersedes \eqref{HBbound} if
$$
\frac{2r}{5(r-2)}-\frac{r+2}{2r(r-2)}\le \frac{r^2+5r+2}{4(r^2+r)},
$$
which is the case if $3\le r\le 5$.
Furthermore, we find that 
$$
N^{(17j_2-10)/(2(5j_2-3))}q^{(j_2-1)/(2(5j_2-3))}\leq N^{2-1/r}q^{(r+2)/(4r^2)}\quad  \mbox{iff }
$$
$$
 N\geq q^{(2(j_2-1)r^2-(5j_2-3)(r+2))/(2(3j_2-2)r^2-4(5j_2-3)r)},
$$  
and therefore our second bound in \eqref{maincor} supersedes \eqref{HBbound} if 
\begin{equation} \label{takes}
\frac{2(j_2-1)r^2-(5j_2-3)(r+2)}{2(3j_2-2)r^2-4(5j_2-3)r}\le\frac{r^2+5r+2}{4(r^2+r)}.
\end{equation}
For our particular case when $j_2=4$, this means that
$$
N^{29/17}q^{3/34}\leq N^{2-1/r}q^{(r+2)/(4r^2)}\quad  \mbox{iff }
 N\geq q^{(6r^2-17(r+2))/(20r^2-68r)},
$$  
and our second bound in \eqref{maincor1} supersedes \eqref{HBbound} if 
$$
\frac{6r^2-17(r+2)}{20r^2-68r}\le \frac{r^2+5r+2}{4(r^2+r)},
$$
which is the case if $4\le r\le 17$.
Note that 
$$
\frac{r^2+5r+2}{4(r^2+r)}= \begin{cases} 13/24=0.542... & \mbox{if } r=3\\ 47/153=0.307... & \mbox{if } r=17\end{cases}
$$ 
and $19/35>13/24$. It follows that our bound \eqref{maincor1} for prime power moduli is comparatively stronger than Heath-Brown's bound for squarefree moduli in the range $q^{47/153}\le N\le q^{7/12}$. More in detail, the first bound in \eqref{maincor1} supersedes Heath-Brown's bound for $r=3$ in the range $q^{19/35}\le N\le q^{7/12}$, and the second bound in \eqref{maincor1} supersedes Heath-Brown's bounds for $r=3,4,...,17$ in the range $q^{47/153}\le N\le q^{19/35}$. 

If $j_2=3$, then \eqref{takes} takes the form 
$$
\frac{2r^2-6(r+2)}{7r^2-24r}\le\frac{r^2+5r+2}{4(r^2+r)}.
$$
If this is the case, then our second bound in \eqref{maincor2} supersedes \eqref{HBbound}. This happens if $4\le r\le 25$.
Note that 
$$
\frac{r^2+5r+2}{4(r^2+r)}= \begin{cases} 13/24=0.542... & \mbox{if } r=3\\ 94/325=0.289... & \mbox{if } r=25\end{cases}
$$ 
and $14/25>13/24$. It follows that our bound \eqref{maincor2} for prime power moduli is comparatively stronger than Heath-Brown's bound for squarefree moduli in the range $q^{94/325}\le N\le q^{7/12}$. More in detail, the first bound in \eqref{maincor2} supersedes Heath-Brown's bound for $r=3$ in the range $q^{14/25}\le N\le q^{7/12}$, and the second bound in \eqref{maincor1} supersedes Heath-Brown's bounds for $r=3,4,...,25$ in the range $q^{94/325}\le N\le q^{14/25}$. 

\section{Preliminaries}
Our first proposition describes a well-known mechanism which translates Dirichlet characters modulo prime powers into additive characters. 
\begin{Proposition} \label{additive characters}   Let $p$ be an odd prime and $n\in \mathbb{N}$ with $n\ge 2$. Let $\chi$ be a primitive character modulo $p^n$. Then there exists $a_0\in \mathbb{Z}_p^{\ast}$ such that 
$$
\chi(1+pt) =e\left(\frac{a_0\log_p(1+pt)}{p^n}\right)
$$
for every $t\in \mathbb{Z}$. 
\end{Proposition}
\begin{proof}
This is a consequence of \cite[Lemma 13]{Mil}. 
\end{proof}
Next, we provide a $p$-adic version of Weyl differencing (with short $p$-adic intervals). 
\begin{Proposition} \label{Weyl shift} Let $p$ be a prime and $n\in \mathbb{N}$. Let $C\in \mathbb{R}$, $X\ge 1$ and $H=p^{\kappa}\le X$ with $\kappa\in \mathbb{N}\cup \{0\}$. Let $\Phi:\mathbb{R}\rightarrow \mathbb{R}_{\ge 0}$ be a Schwartz class function with compact support in $[-1,1]$ and $F:\mathbb{Z}\rightarrow \mathbb{R}$ be a function. Then
\begin{equation*} 
\begin{split}
& \left|\sum\limits_{w\in \mathbb{Z}} \Phi\left(\frac{w-C}{X}\right)e\left(\frac{F(w)}{p^n}\right)\right|^2\\
\ll & XH+H\sum\limits_{0<|h|<2X/H}\left| \sum\limits_{w\in \mathbb{Z}} \Phi_h\left(\frac{w-C}{X}\right) e\left(\frac{F(w+p^{\kappa}h)-F(w)}{p^n}\right) \right|,
\end{split}
\end{equation*}
where 
$$
 \Phi_h(y):= \Phi(y) \Phi\left(y+\frac{p^{\kappa}h}{X}\right).
$$
The implied constant above depends only on $\Phi$. 
\end{Proposition}
\begin{proof}
This is a smoothed version of \cite[inequality (37)]{Mil} and can be proved in the same way.
\end{proof}
Further, we will need the Poisson summation formula over residue classes. 
\begin{Proposition} \label{Poisson summation} Let $\Omega:\mathbb{R}\rightarrow \mathbb{C}$ be a Schwartz class function. Let $C\in \mathbb{R}$, $X\ge 1$, $q\in \mathbb{N}$ and $r\in \mathbb{Z}$. Then
$$
\sum\limits_{t\equiv r\bmod{q}} \Omega\left(\frac{t-C}{X}\right)=\frac{X}{q} \sum\limits_{t\in \mathbb{Z}} \hat{\Omega}\left(\frac{tX}{q}\right) e\left(\frac{t(r-C)}{q}\right). 
$$
\end{Proposition}
\begin{proof} This arises by a linear change of variables from the well-known basic version of the Poisson summation formula which states that
$$
\sum\limits_{m\in \mathbb{Z}} F(m)=\sum\limits_{w\in \mathbb{Z}} \hat{F}(m)
$$
for any Schwartz class function $F:\mathbb{R}\rightarrow \mathbb{C}$ (see \cite{StSh}).
\end{proof}
Finally, we will need the following evaluation of complete exponential sums modulo prime powers due to Cochrane and Zheng, which amounts to a $p$-adic version of the stationary phase method.  To formulate this result, we first introduce some notations. For a polynomial $f$ with integer coefficients, we define its order $\text{ord}_p(f)$ modulo $p$ as the largest exponent $k$ such that $p^k$ divides all the coefficients of $f$. If $f=f_1/f_2\in \mathbb{Z}(X)$ with $f_1,f_2\in \mathbb{Z}[X]$ is a rational function over $\mathbb{Z}$, then we define its order modulo $p$ as $\text{ord}_p(f):=\text{ord}_p(f_1)-\text{ord}_p(f_2)$. Set $t:=\text{ord}_p(f')$. A critical point $\alpha$ of $f$ modulo $p$ is a zero of $p^{-t}f'$ over $\mathbb{Z}/p\mathbb{Z}$, i.e. $p^{-t}f'(\alpha)\equiv 0\bmod{p}$. By $\nu_p(f,\alpha)$ we denote its multiplicity. If $\alpha$ is not a critical point of $f$ modulo $p$, then we set $\nu_p(f,\alpha):=0$. 

\begin{Proposition}[Cochrane-Zheng] \label{exp sum estimates}
Let $p$ be a prime and $f=f_1/f_2$ be a nonconstant rational function, where $f_1,f_2\in \mathbb{Z}[X]$. Set $t:=\text{\rm ord}_p(f')$. Suppose that $\alpha\in \mathbb{Z}$ such that $f_2(\alpha)\not\equiv 0\bmod{p}$. Set $\nu:=\nu_p(f,\alpha)$. Let $m$ be a positive integer. Suppose that $m\ge t+2$ if $p$ is odd and $m\ge t+3$ if $p=2$. Set 
$$
S_{\alpha}(f,p^m):=\sum\limits_{\substack{n\bmod{p^m}\\ n\equiv \alpha\bmod{p}}} e\left(\frac{f(n)}{p^m}\right).
$$  
Then $S_\alpha(f,p^m)=0$ if $\nu=0$ and 
$$
\left|S_\alpha(f,p^m)\right|\le \nu p^{t/(\nu+1)}p^{m(1-1/(\nu+1))} 
$$
if $\nu\ge 1$. 
\end{Proposition}

\begin{proof} This is found in \cite[Theorem 3.1]{Coch}.
\end{proof}

\section{Transformation into sums with additive characters} \label{transf}
Fix an odd prime $p$. Let $Q\in \mathbb{Z}_p[X,Y]$ be a quadratic form given as 
$$
Q(X,Y)=aX^2+2bXY+cY^2.
$$
Denote the determinant $\det(Q)=ac-b^2$ of $Q$ by $\Delta$. Throughout the sequel, we assume that $(c\Delta,p)=1$. Let $\Psi,\Phi:\mathbb{R}\rightarrow \mathbb{R}_{\ge 0}$ be Schwartz class functions with compact support in $[-1,1]$. We aim to estimate character sums of the form
\begin{equation} \label{aimsum}
S_Q(\Psi,\Phi,A,B,M,N;\chi):=\sum_{x\in \mathbb{Z}} \sum_{y\in\mathbb{Z}} \Psi\left(\frac{x-A}{M}\right) \Phi\left(\frac{y-B}{N}\right) \chi(Q(x,y)),
\end{equation}
where $A,B\in \mathbb{Z}$, $M,N\in \mathbb{N}$ and $\chi$ is a primitive Dirichlet character modulo $p^n$ with $n\in \mathbb{N}$. We first simplify $Q(X,Y)$ using quadratic completion as follows. Let $\overline{c}\in \mathbb{Z}_p$ such that $c\overline{c}\equiv 1\bmod{p^n}$, where $n\ge 2$. Then we have
$$
Q(X,Y)=c\left(a\overline{c}X^2+2b\overline{c}XY+Y^2\right)=c\left(\left(a\overline{c}-b^2\overline{c}^2\right)X^2+\left(b\overline{c}X+Y\right)^2\right)=c\left(\alpha X^2+Z^2\right),
$$ 
where 
$$
Z:=b\overline{c}X+Y \quad \mbox{and} \quad \alpha:=a\overline{c}-b^2\overline{c}^2=\Delta\overline{c}^2.
$$  
Hence, we may write our character sum in the form
\begin{equation} \label{aimsumtrans}
S_Q(\Psi,\Phi,A,B,M,N;\chi):=\chi(c)\sum_{x\in \mathbb{Z}} \sum_{y\in\mathbb{Z}} \Psi\left(\frac{x-A}{M}\right) \Phi\left(\frac{y-\tilde{B}}{N}\right) \chi(\alpha x^2+y^2)
\end{equation}
with 
$$
\tilde{B}:=b\overline{c}x+B.
$$
Noting that $\chi(\alpha x^2+y^2)=0$ if both $x$ and $y$ are divisible by $p$, we break the above sum into two parts 
$$
S_Q^{1}(\Psi,\Phi,A,B,M,N;\chi):=\chi(c)\sum_{x\in \mathbb{Z}}  \Psi\left(\frac{x-A}{M}\right)\sum_{\substack{y\in\mathbb{Z}\\ (y,p)=1}} \Phi\left(\frac{y-\tilde{B}}{N}\right) \chi(\alpha x^2+y^2)
$$
and 
$$
S_Q^{2}(\Psi,\Phi,A,B,M,N;\chi):=\chi(c\alpha)\sum_{\substack{y\in \mathbb{Z}\\ p|y}} \Phi\left(\frac{y-\tilde{B}}{N}\right) \sum_{\substack{x\in\mathbb{Z}\\ (x,p)=1}} \Psi\left(\frac{x-A}{M}\right)  \chi(\overline{\alpha}y^2+x^2).
$$
To bound $S_Q^1$, we shall estimate the innermost sum over $y$ non-trivially and then sum up the result over $x$ trivially.  In an analogous way, to bound $S_Q^2$, we shall estimate the innermost sum over $x$ nontrivially and then sum up the result over $y$ trivially. The above division into two subsums leaves us with estimating sums over variables which are coprime to $p$. This restriction will be advantageous in our calculations in section \ref{j1j2}. 

In the following, we focus only on the sum $S_Q^1$ since the estimation of $S_Q^2$ will be similar. So we aim to estimate single variable sums of the form
\begin{equation} \label{aimsumnew}
T(\Phi,\tilde{B},N,\beta;\chi):=\sum_{\substack{y\in \mathbb{Z}\\ (y,p)=1}} \Phi\left(\frac{y-\tilde{B}}{N}\right)\chi(\beta+y^2),
\end{equation}
where $\tilde{B},\beta \in \mathbb{Z}$ and $N\in \mathbb{N}$. This is of independent interest. We start with breaking the above sum into residue classes modulo $p$ in the form
\begin{equation}\label{splits}
\begin{split}
 &\sum_{\substack{y\in \mathbb{Z}\\ (y,p)=1}} \Phi\left(\frac{y-\tilde{B}}{N}\right)\chi(\beta+y^2)\\
=&\sum\limits_{\substack{u\bmod{p}\\ u\not\equiv 0,\beta \bmod{p}}}\sum\limits_{\substack{y\in \mathbb{Z}\\ y^2\equiv u-\beta\bmod{p}}}  \Phi\left(\frac{y-\tilde{B}}{N}\right)\chi(\beta+y^2)\\
 =&\sum\limits_{\substack{u\bmod{p}\\ u\not\equiv 0,\beta \bmod{p}}}\sum\limits_{\substack{v \bmod p\\ v^2\equiv u-\beta \bmod {p}}}\sum\limits_{\substack{y\in \mathbb{Z}\\ y\equiv v \bmod {p}}} \Phi\left(\frac{y-\tilde{B}}{N}\right)\chi(\beta+y^2)\\
 =& \sum\limits_{\substack{u\bmod{p}\\ u\not\equiv 0,\beta \bmod{p}}} \sum\limits_{\substack{v \bmod p\\ v^2\equiv u-\beta \bmod {p}}} \sum\limits_{w\in \mathbb{Z}} \Phi\left(\frac{pw+v-\tilde{B}}{N}\right)\chi\left(\beta+\left(pw+v\right)^2\right). 
\end{split}
\end{equation}
Here we note that the summation condition $u\not\equiv \beta\bmod{p}$ ensures that $(v,p)=1$ and the summation condition $u\not\equiv 0\bmod{p}$ can be included because if $u\equiv 0\bmod{p}$ then $\chi(\beta+y^2)=0$. Now,
if $v_0$ is a solution to the quadratic congruence $v_0^2\equiv u-\beta\bmod{p}$, then by Hensel's lemma, it can be uniquely lifted to a solution $v_1$ to the congruence $v_1^2\equiv u-\beta\bmod{p^n}$ satisfying $v_1\equiv v_0\bmod{p}$. Fix $u\in [0,p-1]\cap \mathbb{N}$ such that $u\not\equiv 0,\beta\bmod{p}$ and $v\in \mathbb{Z}$ such that $v^2=u-\beta \bmod{p^n}$, if existent. We are left with estimating the sum
\begin{equation} \label{Sigmasum}
\begin{split}
\Sigma(\Phi,C,X,\beta,u,v;\chi):=& \sum\limits_{w\in \mathbb{Z}} \Phi\left(\frac{pw+v-\tilde{B}}{N}\right)\chi\left(\beta+\left(pw+v\right)^2\right)\\
= & \sum\limits_{w\in \mathbb{Z}}\Phi\left(\frac{w-C}{X}\right)\chi\left(\beta+\left(pw+v\right)^2\right),
\end{split}
\end{equation}
where 
$$
C:=\frac{\tilde{B}-v}{p} \quad \mbox{and} \quad X:=\frac{N}{p}.
$$
Using Proposition \ref{additive characters}, we may write
\begin{equation*}
\begin{split}
\chi\left(\beta+\left(pw+v\right)^2\right)= & \chi\left(\beta+v^2+p(2vw+pw^2)\right)\\
= & \chi(u)\chi\left(1+p\overline{u}(2vw+pw^2)\right)\\
= & \chi(u)e\left(\frac{a_0\log_p(1+p\overline{u}g(w))}{p^n}\right),
\end{split}
\end{equation*}
where $\overline{u}\in \mathbb{Z}_p$ satisfies $u\overline{u}\equiv 1\bmod{p}^n$,
\begin{equation} \label{gdef}
g(w):=2vw+pw^2
\end{equation}
and $a_0\in \mathbb{Z}_p$ with $(a_0,p)=1$ depends only on $\chi$. Hence, we have 
$$
\Sigma(\Phi,C,X,\beta,u,v;\chi)=\chi(u)\sum\limits_{w\in \mathbb{Z}}\Phi\left(\frac{w-C}{X}\right)e\left(\frac{a_0\log_p(1+p\overline{u}g(w))}{p^n}\right).
$$
In the following we abbreviate $\Sigma(\Phi,C,X,\beta,u,v;\chi)$ by $\Sigma$. 

\section{Application of Weyl differencing}
Next we apply Weyl differencing to transform the term $\Sigma$. Set 
\begin{equation} \label{Fdef}
F(w):=a_0\log_p(1+p\overline{u}g(w))
\end{equation}
and $H_1:=p^{k_1}\le X$, where $k_1\in \mathbb{N}\cup \{0\}$, to be fixed later. Then Proposition \ref{Weyl shift} gives 
\begin{equation}
\begin{split}
|\Sigma|^2=& \bigg| \sum\limits_{w\in \mathbb{Z}}\Phi\left(\frac{w-C}{X}\right)e\left(\frac{F(w)}{p^n}\right)\bigg|^2\\
\ll & XH_1+H_1\sum_{0< |h_1|<2X/H_1} \bigg| \sum_{w\in \mathbb{Z}} \Phi_{h_1}\left(\frac{w-C}{X}\right)e\left(\frac{F(w+p^{k_1}h_1)-F(w)}{p^n}\right)\bigg|.
\end{split}
\end{equation}
Let $H_2:=p^{k_2}\le X$, where $k_2\in \mathbb{N}\cup \{0\}$, to be fixed later. A second application of Weyl differencing after using the Cauchy-Schwarz inequality gives 
\begin{align*}
 |\Sigma|^4
\ll & X^2H_1^2+H_1^2\cdot \frac{X}{H_1}\cdot \sum_{0<|h_1|< 2X/H_1}\bigg|  \sum\limits_{w\in \mathbb{Z}}\Phi_{h_1}\left(\frac{w-C}{X}\right)e\left(\frac{F(w+h_1p^{k_1})-F(w)}{p^n}\right)\bigg|^2\\
\ll & X^2H_1^2+XH_1 \sum_{0<|h_1|<2X/H_1} \bigg(XH_2+H_2 \sum_{0<|h_2|< 2X/H_2}\\ 
& \bigg|\sum_{w\in \mathbb{Z}} \Phi_{h_1,h_2}\left(\frac{w-C}{X}\right)e\left(\frac{F(w+p^{k_1}h_1+p^{k_2}h_2)-F(w+p^{k_1}h_1)-F(w+p^{k_2}h_2)+F(w)}{p^n}\right)\bigg|\bigg)\\
\ll& X^2H_1^2+X^3H_2+XH_1H_2\sum_{0<|h_1|< 2X/H_1}\sum_{0<|h_2|< 2X/H_2}\\
&\bigg|\sum_{w\in \mathbb{Z}}  \Phi_{h_1,h_2}\left(\frac{w-C}{X}\right)e\left(\frac{F(w+p^{k_1}h_1+p^{k_2}h_2)-F(w+p^{k_1}h_1)-F(w+p^{k_2}h_2)+F(w)}{p^n}\right)\bigg|,
\end{align*}
where $\Phi_{h_1,h_2}=\left(\Phi_{h_1}\right)_{h_2}$.
The function $F(t)$ has a $p$-adic Taylor series expansion at $w\in \mathbb{Z}_p$ of the form
$$F(w+\Delta)=F(w)+\Delta F'(w)+\frac{\Delta^2}{2!}\cdot F''(w)+\frac{\Delta^3}{3!}\cdot F'''(w)+\cdots
$$
if $|\Delta|_p<1$. (For background on $p$-adic power series expansions, see \cite[Chapter 3]{Kat}, for example.) It follows that
\begin{equation*}
F(w+\Delta_1)-F(w)=\Delta_1\left(F'(w)+\frac{\Delta_1}{2!}\cdot F''(w)+\frac{\Delta_1^2}{3!}\cdot F'''(w)+\cdots\right)
\end{equation*}
and 
\begin{equation*}
F(w+\Delta_1+\Delta_2)-F(w+\Delta_1)-F(w+\Delta_2)+F(w)=\Delta_1\Delta_2\left(F''(w)+\frac{\Delta_1+\Delta_2}{2}\cdot F'''(w)+\cdots\right)
\end{equation*}
if $|\Delta_1|_p,|\Delta_2|_p< 1$. Taking $\Delta_1:=p^{k_1}h_1$ and $\Delta_2:=p^{k_2}h_2$ with $k_1,k_2\ge 1$, we deduce that
\begin{equation*}
F(w+p^{k_1}h_1)-F(w)=p^{k_1}h_1\left(F'(w)+p^{k_1}G_{h_1}(w)\right)
\end{equation*}
and 
\begin{equation*}
F(w+p^{k_1}h_1+p^{k_2}h_2)-F(w+p^{k_1}h_1)-F(w+p^{k_2}h_2)+F(w)=p^{k_1+k_2}h_1h_2\left(F''(w)+p^kG_{h_1,h_2}(w)\right),
\end{equation*}
where  $k=\min\{k_1,k_2\}$,
$$
G_{h_1}(w)=\frac{1}{2}\cdot F''(w)+\frac{p^{k_1}h_1}{6}\cdot F'''(w)+\frac{p^{2k_1}h_1^2}{24}\cdot F^{(4)}(w)+\cdots
$$
and 
\begin{equation*}
\begin{split}
G_{h_1,h_2}(w)= & \left(\frac{p^{k_1-k}h_1}{2}+\frac{p^{k_2-k}h_2}{2}\right)F'''(w)+\\ & \left(\frac{p^{2k_1-k}h_1^2}{6}+\frac{p^{k_1+k_2-k}h_1h_2}{4}+\frac{p^{2k_2-k}h_2^2}{6}\right)F^{(4)}(w)+\cdots.
\end{split}
\end{equation*}
We note that for $t\ge 1$, the derivatives $F^{(t)}(w)$ are rational functions of non-negative orders over $\mathbb{Z}_p$. So if $k_1,k_2\ge 1$, then $G_{h_1}$ and $G_{h_1,h_2}$ become rational functions of non-negative orders over $\mathbb{Z}_p/p^n\mathbb{Z}_p$ since we may cut off the series expansions defining $G_{h_1}$ and $G_{h_1,h_2}$ at suitable points. 

Combining everything in this section, we obtain 
\begin{equation*}
|\Sigma|^2\ll XH_1+H_1\sum_{0< |h_1|<2X/H_1} \bigg| \sum_{w\in \mathbb{Z}} \Phi_{h_1}\left(\frac{w-C}{X}\right)e\left(\frac{h_1\left(F'(w)+p^{k_1}G_{h_1}(w)\right)}{p^{n-k_1}}\right)\bigg|
\end{equation*}
and 
\begin{equation*}
\begin{split}
|\Sigma|^4\ll & X^2H_1^2+X^3H_2+XH_1H_2\times\\ &\sum_{0<|h_1|<2X/H_1}\sum_{0<|h_2|<2X/H_2}
\left|\sum_{w\in \mathbb{Z}} \Phi_{h_1,h_2}\left(\frac{w-C}{X}\right)e\left(\frac{h_1h_2\left(F''(w)+p^kG_{h_1,h_2}(w)\right)}{p^{n-(k_1+k_2)}}\right)\right|.
\end{split}
\end{equation*}
We further divide the $h_i$-summations according to $l_i=\mbox{ord}_p(h_i)$. Writing $h_i=p^{l_i}g_i$ with $(g_i,p)=1$, we have 
 \begin{equation} \label{Weyl1}
|\Sigma|^2\ll XH_1+H_1\sum_{l_1\ge 0} \sum_{\substack{0< |g_1|<2X/(H_1p^{l_1})\\ (g_1,p)=1}} 
\left|T(l_1,g_1)\right|
\end{equation}
with 
$$
T(l_1,g_1):=\sum_{w\in \mathbb{Z}}\Phi_{p^{l_1}g_1}\left(\frac{w-C}{X}\right)e\left(\frac{g_1\left(F'(w)+p^{k_1}G_{p^{l_1}g_1}(w)\right)}{p^{n-(k_1+l_1)}}\right)
$$
and 
\begin{equation} \label{Weyl2}
\begin{split}
|\Sigma|^4\ll & X^2H_1^2+X^3H_2+ XH_1H_2\times \\ & \sum_{l_1\geq 0}\sum_{l_2\geq 0}
\sum\limits_{\substack{0<|g_1|< 2X/(H_1p^{l_1})\\(g_1,p)=1}}\sum\limits_{\substack{0<|g_2|< 2X/(H_2p^{l_2})\\(g_2,p)=1}} 
\left|T(l_1,l_2,g_1,g_2)\right|
\end{split}
\end{equation}
with 
$$
T(l_1,l_2,g_1,g_2):=\sum_{w\in \mathbb{Z}} \Phi_{p^{l_1}g_1,p^{l_2}g_2}\left(\frac{w-C}{X}\right)e\left(\frac{g_1g_2(F''(w)+p^kG_{p^{l_1}g_1,p^{l_2}g_2}(w))}{p^{n-(k_1+k_2+l_1+l_2)}}\right).
$$ 
Next we apply the Poisson summation formula to transform the sums $T(l_1,g_1)$ and $T(l_1,l_2,g_1,g_2)$. 

\section{Application of Poisson summation}
In the following, we set 
\begin{equation} \label{f1def}
f_1(w):=g_1\left(F'(w)+p^{k_1}G_{p^{l_1}g_1}(w)\right), \quad \Omega_1:= \Phi_{p^{l_1}g_1}, \quad s_1:=n-(k_1+l_1).
\end{equation}
We write
\begin{equation} \label{T1}
\begin{split}
T(l_1,g_1)= & \sum_{w\in \mathbb{Z}} \Omega_1\left(\frac{w-C}{X}\right)e\left(\frac{f_1(w)}{p^{s_1}}\right)\\
= & \sum\limits_{r \bmod{p^{s_1}}} e\left(\frac{f_1(r)}{p^{s_1}}\right)\sum\limits_{t_1\equiv r\bmod{p^{s_1}}} \Omega_1\left(\frac{t_1-C}{X}\right).
\end{split}
\end{equation}
Now the Poisson summation formula, Proposition \ref{Poisson summation}, gives
\begin{equation*}
\begin{split}
\sum\limits_{t_1\equiv r\bmod{p^{s_1}}} \Omega_1\left(\frac{t_1-C}{X}\right)
= & \frac{X}{p^{s_1}} \sum\limits_{t_1\in \mathbb{Z}} \hat{\Omega}_1\left(\frac{t_1X}{p^{s_1}}\right) e\left(-\frac{t_1C}{p^{s_1}}\right) e\left(\frac{t_1r}{p^{s_1}}\right).
\end{split}
\end{equation*}
Plugging this into \eqref{T1} and interchanging summations, we get
\begin{equation} \label{TT1}
T(l_1,g_1)= \frac{X}{p^{s_1}} \sum_{t_1\in \mathbb{Z}} \hat{\Omega}_1\left(\frac{t_1X}{p^{s_1}}\right) e\left(-\frac{t_1C}{p^{s_1}}\right)
\sum\limits_{r \bmod{p^{s_1}}} e\left(\frac{f_1(r)+t_1r}{p^{s_1}}\right).
\end{equation}
Similarly, we obtain
\begin{equation} \label{TT2}
T(l_1,l_2,g_1,g_2)= \frac{X}{p^{s_2}} \sum_{t_2\in \mathbb{Z}} \hat{\Omega}_2\left(\frac{t_2X}{p^{s_2}}\right) e\left(-\frac{t_2C}{p^{s_2}}\right)\sum\limits_{r \bmod{p^{s_2}}} e\left(\frac{f_2(r)+t_2r}{p^{s_2}}\right),
\end{equation}
where 
\begin{equation} \label{f2def}
f_2(w):=g_1g_2(F''(w)+p^kG_{p^{l_1}g_1,p^{l_2}g_2}(w)), \quad \Omega_2:= \Phi_{p^{l_1}g_1,p^{l_2}g_2}, \quad s_2:=n-(k_1+k_2+l_1+l_2),
\end{equation}
respectively. 
Again, we point out that $f_1(w)$ and $f_2(w)$ define rational functions over $\mathbb{Z}_p/p^n\mathbb{Z}_p$ and hence over $\mathbb{Z}_p/p^{s_1}\mathbb{Z}_p$ and $\mathbb{Z}_p/p^{s_2}\mathbb{Z}_p$, respectively. 

\section{Evaluation of complete exponential sums}
Next, we shall apply Proposition \ref{exp sum estimates} to estimate the complete exponential sums over $r \bmod{p^{s_i}}$, $i=1,2$ in \eqref{TT1} and \eqref{TT2}.
Recall the definitions of the functions $g$, $F$, $f_1$ and $f_2$ in \eqref{gdef}, \eqref{Fdef}, \eqref{f1def} and \eqref{f2def}. We calculate that 
$$
F'(w)=\frac{a_0pg'(w)}{u+pg(w)}=\frac{2a_0p(v+pw)}{u+2v pw+p^2w^2}
\equiv \frac{2a_0p(v+pw)}{\beta+(v+pw)^2} \bmod{p^n}.
$$
For $i=1,2$, we set
\begin{equation} \label{Fidefi}
F_i(r):=f_i(r)+t_ir, 
\end{equation}
which define rational functions over $\mathbb{Z}_p/p^{n}\mathbb{Z}_p$.
Recall the notions introduced before the statement of Proposition \ref{exp sum estimates}. To apply this proposition, we need to determine the multiplicities of possible critical points of $F_i(r)$. We have 
$$
F_1'(r)=f_1'(r)+t_1=g_1\left(F''(r)+p^{k_1}\tilde{G}_1'(r)\right)+t_1
$$
and 
$$
F_2'(r)=f_2'(r)+t_2=g_1g_2\left(F'''(r)+p^k\tilde{G}_2'(r)\right)+t_2,
$$
where we set $\tilde{G}_1:=G_{p^{l_1}g_1}$ and $\tilde{G}_2:=G_{p^{l_1}g_1,p^{l_2}g_2}$. Set 
$$
T(r):=v+pr.
$$
Then the second and third derivatives of $F$ are calculated as 
$$
F''(r)\equiv 2a_0p^2\cdot \frac{\beta-T(r)^2}{\left(\beta+T(r)^2\right)^2}\bmod{p^n}
$$
and 
$$
F'''(r)\equiv -4a_0p^3\cdot \frac{T(r)\left(3\beta-T(r)^2\right)}{\left(\beta+T(r)^2\right)^3}\bmod{p^n}.
$$
It follows that
$$
F_1'(r)\equiv g_1\left(2a_0p^2\cdot \frac{\beta-T(r)^2}{\left(\beta+T(r)^2\right)^2}+p^{k_1}\tilde{G}_1'(r)\right)+t_1
\bmod{p^n}
$$
and 
$$
F_2'(r)\equiv g_1g_2\left(-4a_0p^3\cdot \frac{T(r)\left(3\beta-T(r)^2\right)}{\left(\beta+T(r)^2\right)^3}+p^k\tilde{G}_2'(r)\right)+t_2\bmod{p^n}.
$$
For $i=1,2$, let $\omega_i:=\mbox{ord}_p(F_i')$ and suppose that the multiplicities of the critical points of $F_i$ are bounded by $j_i-1$.  In the last subsection, we will see that $\omega_1\le 5$ if $k_1\ge 6$ and $\omega_2\le 8$ if $k\ge 9$. 
Now it follows from Proposition \ref{exp sum estimates}, \eqref{TT1}, \eqref{TT2} and \eqref{Fidefi} that 
\begin{equation} \label{TT1esti}
T(l_1,g_1)\ll \frac{X}{p^{s_1}} \sum_{t_1\in \mathbb{Z}} \hat{\Omega}_1\left(\frac{t_1X}{p^{s_1}}\right) p^{\omega_1/j_1}p^{(1-1/j_1)s_1} 
\ll \left(\frac{X}{p^{s_1/j_1}}+p^{(1-1/j_1)s_1}\right)
\end{equation}
and 
\begin{equation} \label{TT2esti}
T(l_1,l_2,g_1,g_2)= \frac{X}{p^{s_2}} \sum_{t_2\in \mathbb{Z}} \hat{\Omega}_2\left(\frac{t_2X}{p^{s_2}}\right) p^{\omega_2/j_2}p^{(1-1/j_2)s_2}
\ll \left(\frac{X}{p^{s_2/j_2}}+p^{(1-1/j_2)s_2}\right).
\end{equation}
Under the above-mentioned conditions $k_1\ge 6$ and $k\ge 9$, respectively, we will determine $j_1$ and $j_2$ in the last section as well.  The next section provides a proof of Theorem \ref{mainresult} assuming the validity of the estimates \eqref{TT1esti} and \eqref{TT2esti} above.  
  
\section{Final bounds}
Recalling that $s_1=n-(k_1+l_1)$ and $H_1=p^{k_1}\le X$ and combining \eqref{Weyl1} and \eqref{TT1esti}, we obtain
 \begin{equation} \label{Sigma1esti} 
\begin{split}
|\Sigma|^2\ll & XH_1+H_1\sum_{l_1\ge 0} \sum_{\substack{0< |g_1|<2X/(H_1p^{l_1})\\ (g_1,p)=1}} 
 \left(\frac{X}{p^{n/j_1}}\cdot p^{(k_1+l_1)/j_1}+p^{(1-1/j_1)(n-(k_1+l_1))}\right)\\
\ll & XH_1+\frac{X^2H_1^{1/j_1}}{q^{1/j_1}}+\frac{Xq^{1-1/j_1}}{H_1^{1-1/j_1}},
\end{split}
\end{equation}
where we set 
$$
q:=p^{n}.
$$
If $X\gg q^{(j_1-1)/(2j_1-1)}$, we may choose $H_1$ in such a way that 
$$
q^{(j_1-1)/(2j_1-1)}\le H_1 < pq^{(j_1-1)/(2j_1-1)},
$$
balancing the first and third terms in the last line of \eqref{Sigma1esti}. We note that the condition $k_1\ge 6$ is satisfied if $n>n_0$ for some large enough absolute constant $n_0>0$. Under the above choice of $H_1$, we obtain
$$
\Sigma\ll X^{1/2}q^{(j_1-1)/(2(2j_1-1))}+Xq^{-j_1/(4(2j_1-1))}.
$$ 
Noting \eqref{aimsum}, \eqref{aimsumtrans}, \eqref{aimsumnew}, \eqref{splits}, \eqref{Sigmasum} and $X=N/p$, this implies the bound 
\begin{equation*} 
S_Q^1(\Psi,\Phi,A,B,M,N;\chi)\ll MN^{1/2}q^{(j_1-1)/(2(2j_1-1))}+MNq^{-j_1/(4(2j_1-1))}
\end{equation*}
if $N\ge q^{(j_1-1)/(2j_1-1)}$. We observe that if $N\le q^{(3j_1-2)/(2(2j_1-1))}$, then the first term on the right-hand side dominates. Hence, we have 
\begin{equation} \label{final1}
S_Q^1(\Psi,\Phi,A,B,M,N;\chi)\ll MN^{1/2}q^{(j_1-1)/(2(2j_1-1))} \quad \mbox{if } q^{(j_1-1)/(2j_1-1)}\le N\le q^{(3j_1-2)/(2(2j_1-1))}.
\end{equation}
Similarly, we obtain 
\begin{equation} \label{final11}
S_Q^2(\Psi,\Phi,A,B,M,N;\chi)\ll NM^{1/2}q^{(j_1-1)/(2(2j_1-1))} \quad \mbox{if } q^{(j_1-1)/(2j_1-1)}\le M\le q^{(3j_1-2)/(2(2j_1-1))}.
\end{equation}

Recalling that $s_2=n-(k_1+k_2+l_1+l_2)$, $H_1=p^{k_1}\le X$ and $H_2=p^{k_2}\le X$, and combining \eqref{Weyl2} and \eqref{TT2esti}, we obtain
\begin{equation} \label{Sigma2esti}
\begin{split}
|\Sigma|^4\ll & X^2H_1^2+X^3H_2+ XH_1H_2\times\\ & \sum_{l_1\geq 0}\sum_{l_2\geq 0}
\sum\limits_{\substack{0<|g_1|< 2X/(H_1p^{l_1})\\(g_1,p)=1}} \sum\limits_{\substack{0<|g_2|< 2X/(H_2p^{l_2})\\(g_2,p)=1}} \left(\frac{X}{p^{n/j_2}}\cdot p^{(k_1+k_2+l_1+l_2)/j_2}+p^{(1-1/j_2)(n-(k_1+k_2+l_1+l_2))}\right)\\
\ll &X^2H_1^2+X^3H_2+ \frac{X^4(H_1H_2)^{1/j_2}}{q^{1/j_2}}+\frac{X^3q^{(1-1/j_2)}}{(H_1H_2)^{1-1/j_2}}. 
\end{split}
\end{equation}
If $X\gg q^{(j_2-1)/(3j_2-2)}$, we may choose $H_1$ and $H_2$ in such a way that 
$$
X^{(2j_2-1)/(5j_2-3)}q^{(j_2-1)/(5j_2-3)}\le H_1<pX^{(2j_2-1)/(5j_2-3)}q^{(j_2-1)/(5j_2-3)}
$$
and 
$$
X^{-(j_2-1)/(5j_2-3)}q^{2(j_2-1)/(5j_2-3)}\le H_2<pX^{-(j_2-1)/(5j_2-3)}q^{2(j_2-1)/(5j_2-3)},
$$ 
balancing the first, second and fourth terms in the last line of \eqref{Sigma2esti}. We note that the condition $k\ge 9$ is satisfied if $n>n_0'$ for some large enough absolute constant $n_0'>0$. Under the above choice of $H_1$ and $H_2$, we obtain
$$
\Sigma \ll X^{(7j_2-4)/(2(5j_2-3))}q^{(j_2-1)/(2(5j_2-3))}+X^{(20j_2-11)/(4(5j_2-3))}q^{-j_2/(2(5j_2-3))}.
$$
This implies the bound 
\begin{equation*} 
S_Q^1(\Psi,\Phi,A,B,M,N;\chi)\ll MN^{(7j_2-4)/(2(5j_2-3))}q^{(j_2-1)/(2(5j_2-3))}+MN^{(20j_2-11)/(4(5j_2-3))}q^{-j_2/(2(5j_2-3))}
\end{equation*}
if $N\ge q^{(j_2-1)/(3j_2-2)}$. We observe that if $N\le q^{2/3}$, then the first term on the right-hand side dominates. Hence, we have 
\begin{equation} \label{final2} 
S_Q^1(\Psi,\Phi,A,B,M,N;\chi)\ll MN^{(7j_2-4)/(2(5j_2-3))}q^{(j_2-1)/(2(5j_2-3))} \quad \mbox{if } q^{(j_2-1)/(3j_2-2)}\le N\le q^{2/3}.
\end{equation}
Similarly, we obtain
\begin{equation} \label{final22} 
S_Q^2(\Psi,\Phi,A,B,M,N;\chi)\ll NM^{(7j_2-4)/(2(5j_2-3))}q^{(j_2-1)/(2(5j_2-3))} \quad \mbox{if } q^{(j_2-1)/(3j_2-2)}\le M\le q^{2/3}.
\end{equation}
This completes the proof of Theorem \ref{mainresult}.  

\section{Analysis of orders and multiplicities } \label{j1j2} We recall that  
$$
F_1'(r)\equiv R_1(r)+g_1p^{k_1}\tilde{G}_1'(r)\bmod{p^n}
$$
and 
$$
F_2'(r)\equiv R_2(r)+g_1g_2p^k\tilde{G}_2'(r)\bmod{p^n},
$$
where 
$$
R_1(r):=2g_1a_0p^2\cdot \frac{\beta-T(r)^2}{\left(\beta+T(r)^2\right)^2}+t_1,
$$
$$
R_2(r):=-4g_1g_2a_0p^3\cdot \frac{T(r)\left(3\beta-T(r)^2\right)}{\left(\beta+T(r)^2\right)^3}+t_2
$$
with
$$
T(r)=v+pr.
$$ 
We also recall that $p>3$ and   
$$
(a_0g_1g_2uv,p)=1, \quad \mbox{where } u\equiv \beta+v^2 \bmod p^n. 
$$
For $i=1,2$, we aim to obtain upper bounds for 
$$
\omega_i=\mbox{ord}_p(F_i')
$$
and the maximum $m_i$ of the multiplicities of critical points of $F_i$. We shall establish below that $\mbox{ord}_p(R_1)\le 5$ and $\mbox{ord}_2(R_2)\le 8$. Hence, if $k_1\ge 6$ then $\omega_1=$ord$_p(R_1)$, and if $k\ge 9$ then $\omega_2=\mbox{ord}_p(R_2)$. Moreover, in these situations, $m_i$ equals the maximum of the multiplicities of roots of $R_i$ for $i=1,2$. Thus we may discard the higher-order terms $g_1p^{k_1}\tilde{G}_1'(r)$ and $g_1g_2p^k\tilde{G}_2'(r)$ and are left with investigating the rational functions $R_1(r)$ and $R_2(r)$. (In fact, a more careful analysis of the functions $\tilde{G}_1$ and $\tilde{G}_2$ reveals that the conditions $k_1\ge 1$ and $k\ge 1$ suffice to get rid of the terms $g_1p^{k_1}\tilde{G}_1'(r)$ and $g_1g_2p^k\tilde{G}_2'(r)$, but this does not change our main result.) 

Starting with $R_1(r)$, we calculate that  
$$
R_1(r)=\frac{P_1(r)}{Q_1(r)},
$$
where 
$$
Q_1(r):=(\beta+(v+pr)^2)^2
$$
and 
$$
P_1(r)=c_0+c_1r+c_2r^2+c_3r^3+c_4r^4
$$
with
\begin{equation*}
\begin{split}
c_0:= & 2p^2a_0g_1(\beta-v^2)+t_1(\beta+v^2)^2=2p^2a_0g_1(\beta-v^2)+t_1u^2\\
c_1:= & 4p(-p^2a_0g_1v+\beta t_1 v+t_1v^3)=4p(-p^2a_0g_1v+uvt_1)\\
c_2:= & 2p^2(-p^2a_0g_1+\beta t_1+ 3t_1v^2)\\
c_3:= & 4p^3t_1v\\
c_4:= & p^4t_1.
\end{split}
\end{equation*}
Now we consider the following cases, assuming that $k_1\ge 6$.  \\ \\
{\bf Case 1.1:} {ord$_p(t_1)\le 1$}. In this case, we have 
$$
\mbox{ord}_p(c_1)= 1+\mbox{ord}_p(t_1)<\mbox{ord}_p(c_i) \quad \mbox{for } i=2,3,4.
$$
It follows that $\omega_1\le 2$,
and the reduced polynomial $P_1(r)/p^{\omega_1}$ is at most linear modulo $p$. Hence, $m_1\le 1$.\\ \\ 
{\bf Case 1.2:} {ord$_p(t_1)\ge 3$.} 
In this case, we have 
$$
\mbox{ord}_p(c_1)=3<\mbox{ord}_p(c_i) \quad \mbox{for } i=2,3,4.
$$ 
It follows that $\omega_1\le 3$,
and again, the reduced polynomial $P_1(r)/p^{\omega_1}$ is at most linear modulo $p$. Hence, $m_1\le 1$.\\ \\
{\bf Case 1.3:} {$\mbox{ord}_p(t_1)=2$.} In this case, we have
$$
\mbox{ord}_p(c_3)=5<6=\mbox{ord}_p(c_4)
$$ 
and hence, the reduced polynomial $P_1(r)/p^{\omega_1}$ is at most cubic modulo $p$. Below we prove by contradiction that it is in fact at most quadratic modulo $p$. 
Suppose it is cubic modulo $p$. Then we have 
$$
\mbox{ord}_p(c_i)\ge \mbox{ord}_p(c_3)=5 \quad \mbox{for } i=0,1,2
$$ 
and it follows that
\begin{equation*}
\begin{cases}
  \mbox{ord}_p(c_2/p^4)=\mbox{ord}_p(-a_0g_1+\beta t_1'+3t_1'v^2)\geq  1 \\
  \mbox{ord}_p(c_1/p^3)=\mbox{ord}_p(-a_0g_1v+\beta t_1'v)\geq  2\\
  \mbox{ord}_p(c_0/p^2)=\mbox{ord}_p(2a_0g_1(\beta -v^2)+t_1' (\beta+v^2)^2)\geq  3,
\end{cases}
\end{equation*}
where $t_1':=t_1/p^2$. Multiplying the term $-a_0g_1+\beta t_1'+3t_1'v^2$ in the first line above by $v$ and subtracting the term $-a_0g_1v+\beta t_1'v$ in the second line, we deduce that
$$
\mbox{ord}_p(3t_1'v^3)\geq 1.
$$
This implies that $\mbox{ord}_p(v)\geq 1$, which is a contradiction to the coprimality of $v$ and $p$. We conclude that the reduced polynomial $P_1(r)/p^{\omega_1}$ is at most quadratic modulo $p$, as claimed. Hence, $m_1\le 2$. Moreover, $\omega_1\le \mbox{ord}_p(c_3)=5$.\\ \\
So, in all three cases above, we have
$$
m_1\le 2 \quad \mbox{and} \quad  
\omega_1\le 5.
$$
Since $m_1\le 2$, we can take $j_1=2+1=3$ in \eqref{TT1esti}.\\

Turning to $R_2(r)$, we calculate that  
$$
R_2(r)=\frac{P_2(r)}{Q_2(r)},
$$
where 
$$
Q_2(r):=\left(\beta+(v+pr)^2\right)^3
$$
and 
$$
P_2(r)=d_0+d_1r+d_2r^2+d_3r^3+d_4r^4+d_5r^5+d_6r^6
$$
with
\begin{equation}\label{coeq}
\begin{split}
    d_0:=& 4p^3a_0g_1g_2v(v^2-3\beta)-t_2(\beta+v^2)^3=a_0g_1g_2p^3v(v^2-3\beta)-t_2u^3\\
    d_1:=& 6p(\beta^2t_2v+2\beta t_2 v^3+t_2v^5+2p^3a_0g_1g_2(v^2-\beta))\\
    d_2:=& 3p^2(\beta^2t_2+6\beta t_2v^2+5t_2v^4+4p^3a_0g_1g_2v)\\
    d_3:=& 4p^3(p^3a_0g_1g_2+3\beta t_2v+5t_2v^3)\\
    d_4:=& 3p^4(5v^2+\beta)t_2\\
    d_5:=& 6p^5t_2v\\
    d_6:=& p^6t_2.
\end{split}
\end{equation}
Again, we consider three cases below, assuming that $k\ge 9$.\\ \\
{\bf Case 2.1:} {$\mbox{ord}_p(t_2)\le 2$.}
In this case, we have
$$\mbox{ord}_p(d_0)=\mbox{ord}_p(t_2)<\mbox{ord}_p(d_i) \quad \mbox{for } i=1,...,6
$$
It follows that $\omega_2\le 2$ and the reduced polynomial $P_2(r)/p^{\omega_2}$ is constant but non-zero modulo $p$. Hence, there is no critical point. \\ \\
{\bf Case 2.2:} {$\mbox{ord}_p(t_2)\ge 4$.} In this case, we have 
$$
\mbox{ord}_p(d_2)=5<\mbox{ord}_p(d_i) \quad \mbox{for } i=3,...,6. 
$$
It follows that $\omega_2\le 5$ and the reduced polynomial $P_2(r)/p^{\omega_2}$ in at most quadratic. Hence, $m_2\le 2$.  \\ \\
{\bf Case 2.3:} {$\mbox{ord}_p(t_2)=3$.} In this case, we have 
$$
\mbox{ord}_p(d_5)=8<9=\mbox{ord}_p(d_6)
$$ 
and hence, the reduced polynomial $P_2(r)/p^{\omega_2}$ is at most quartic or pentic modulo $p$. Below we prove by contradiction that it is in fact at most cubic modulo $p$. 
Suppose it is quartic or pentic modulo $p$. Then we have 
$$
\mbox{ord}_p(d_i)\ge \min\{\mbox{ord}_p(d_4),\mbox{ord}_p(d_5)\}\ge 7 \quad \mbox{for } i=0,1,2,3
$$ 
and it follows that
\begin{equation}\label{cseq}
\begin{cases}
   \mbox{ord}_p(d_3/p^6)=\mbox{ord}_p(a_0g_1g_2+3\beta t_2'v+5t_2'v^3)\geq 1\\
   \mbox{ord}_p(d_2/p^5)=\mbox{ord}_p(\beta^2t_2'+6\beta t_2'v^2+5t_2'v^4+4a_0g_1g_2v)\geq 2\\
   \mbox{ord}_p(d_1/p^4)=\mbox{ord}_p(\beta^2t_2'v+2\beta t_2' v^3+t_2'v^5+2a_0g_1g_2(v^2-\beta))\geq 3\\
   \mbox{ord}_p(d_0/p^3)=\mbox{ord}_p(4a_0g_1g_2v(v^2-3\beta)-t_2'u^3)\geq 4,
\end{cases}
\end{equation}
where $t_2':=t_2/p^3$.
In particular, this implies 
\begin{equation}\label{nwweq}
\begin{cases}
a_0g_1g_2+3\beta t_2' v+5t_2' v^3\equiv 0 \bmod{p}\\
\beta^2t_2'+6\beta t_2'v^2+5t_2'v^4+4a_0g_1g_2 v\equiv 0\bmod{p}\\
\beta^2t_2'v +2\beta t_2' v^3+t_2'v^5+2a_0g_1g_2(v^2-\beta)\equiv 0 \bmod{p}\\
4a_0g_1g_2v^3-12a_0g_1g_2\beta v-t_2'u^3\equiv 0 \bmod{p}.
\end{cases}
\end{equation}
From the first congruence above, we infer that
\begin{equation}\label{eva}
a_0g_1g_2\equiv -3\beta t_2' v-5t_2' v^3\bmod{p}.
\end{equation}
Substituting the right-hand side for the term $a_0g_1g_2$ in the second and third equations in \eqref{nwweq} and dividing by $t_2'$, we get
\begin{equation}\label{one}
\beta^2-6\beta v^2- 15 v^4 \equiv 0 \bmod{p}
\end{equation}
and
\begin{equation*}
    \begin{split}
  7\beta^2+6\beta v^2-9v^4+6\beta^2\equiv 0 \bmod{p}.
 \end{split}
\end{equation*}
Subtracting 7 times the first congruence above from the second one and dividing by 6, we obtain
\begin{equation*}
(\beta+4v^2)^2\equiv 0 \bmod{p}
\end{equation*}
which implies
\begin{equation*}
\beta\equiv -4v^2 \bmod{p}.
\end{equation*}
Now substituting $-4v^2$ for $\beta$ in \eqref{one} gives
$$
\beta^2-6\beta v^2-15v^4\equiv 25v^2\equiv 0 \bmod{p}.
$$
This contradicts the coprimality of $v$ and $p$. We conclude that the reduced polynomial $P_2(r)/p^{\omega_2}$ is at most cubic modulo $p$, as claimed. Hence, $m_2\le 3$. Moreover, $\omega_2\le \mbox{ord}_p(c_5)=8$.\\ \\
So, in all three cases above, we have 
$$
m_2\le 3 \quad \mbox{and} \quad  \omega_2\le 8.
$$
Since $m_2\le 3$, we can take $j_2=3+1=4$ in \eqref{TT2esti}.\\ 

Below we will show that under a certain congruence condition on $p$, we can reduce our bound for the degree of $P_2(r)/p^{\omega_2}$ modulo $p$ further.  In cases 2.1 and 2.2 above, we found that this reduced polynomial $P_2(r)/p^{\omega_2}$ is at most quadratic modulo $p$, so it suffices to look at case 2.3 in which $\mbox{ord}_p(t)=3$. Suppose, $P_2(r)/p^{\omega_2}$ is cubic modulo $p$. Then
$$
\mbox{ord}_p(d_i)\ge \mbox{ord}_p(d_3)\ge 6 \quad \mbox{for } i=0,1,2
$$ 
and it follows that
\begin{equation*}
\begin{cases}
  \mbox{ord}_p(d_2/p^5)= \mbox{ord}_p(\beta^2t_2'+6\beta t_2'v^2+5t_2'v^4+4a_0g_1g_2v)\geq 1\\
  \mbox{ord}_p(d_1/p^4)= \mbox{ord}_p(\beta^2t_2'v+2\beta t_2' v^3+t_2'v^5+2a_0g_1g_2(v^2-\beta))\geq 2\\
  \mbox{ord}_p(d_0/p^3)= \mbox{ord}_p(4a_0g_1g_2v(v^2-3\beta)-t_2'u^3)\geq 3,
\end{cases}
\end{equation*}
where again $t_2':=t/p^3$. In particular, this implies 
\begin{equation}\label{enw}
\begin{cases}
\beta^2t_2'+6\beta t_2' v^2+5t_2'v^4+4a_0g_1g_2 v\equiv 0\bmod{p}\\
\beta^2t_2'v +2\beta t_2' v^3+t_2'v^5+2a_0g_1g_2(v^2-\beta)\equiv 0 \bmod{p}\\
4a_0g_1g_2v(v_2-3\beta)-t_2'u^3\equiv 0 \bmod{p}.\\
\end{cases}
\end{equation}
Multiplying the first congruence above by $v$, subtracting the second, and dividing by $2u=2(\beta+v^2)$, we obtain
\begin{equation*}
    a_0g_1g_2\equiv -2t_2'v^3 \bmod{p}.
\end{equation*}
Substituting the right-hand side for $a_0g_1g_2$ in the first congruence in \eqref{enw}, and dividing by $t_2'$, we get
\begin{equation*}
 \beta^2+6\beta v^2 -3v^4 \equiv 0 \bmod{p},
\end{equation*}
which implies 
$$
\left(\frac{\beta+3v^2}{2v^2}\right)^2\equiv 3 \mod{p}.
$$
Hence, $3$ is a quadratic residue modulo $p$. So if 
$$
\left(\frac{3}{p}\right)=-1, 
$$
then $P_2(r)/p^{\omega_2}$ is at most quadratic modulo $p$. Using quadratic reciprocity, it is easy to see that this happens if and only if $p\equiv \pm 5 \bmod{12}$. Hence, in this case, we deduce that $m_2\le 2$ and we can therefore take $j_2=2+1=3$ in \eqref{TT2esti}.

\end{document}